\documentclass[reqno,12pt,twoside]{amsart}
\usepackage{amsmath,amssymb,amsfonts,amsthm}
\usepackage{enumerate}

\usepackage[backref]{hyperref}
\usepackage{microtype}
\newtheorem{theorem}{Theorem}[section]
\theoremstyle{plain}
\newtheorem{corollary}[theorem]{Corollary}
\newtheorem{definition}{Definition}[section]
\newtheorem{example}[theorem]{Example}
\newtheorem{lemma}[theorem]{Lemma}
\newtheorem{proposition}[theorem]{Proposition}
\numberwithin{equation}{section}
\newtheorem{remark}[theorem]{Remark}

\numberwithin{equation}{section}
\newcommand{\RR}{\mathbb{R}}
\newcommand{\CC}{\mathbb{C}}
\newcommand{\NN}{\mathbb{N}}

\newcommand{\bba}{\bar{\beta}}

\newcommand{\gba}{\bar{g}}

\newcommand{\wba}{\bar{w}}

\newcommand{\zba}{\bar{z}}
\newcommand{\Zba}{\overline{Z}}

\newcommand{\lc}{\mathcal{X}}
\newcommand{\tfour}{{D^{\mathrm{IV}}_4}}

\newcommand{\tfourM}{{D^{\mathrm{IV}}_{m}}}

\DeclareMathOperator{\Aut}{Aut}

\newcommand{\kba}{\bar{k}}
\newcommand{\lba}{\bar{\lambda}}

\newcommand{\rk}{\operatorname{rk}}
\renewcommand{\Re}{\operatorname{Re}}
\renewcommand{\Im}{\operatorname{Im}}

\begin{document}
\title[Maps from the sphere into the tube over the future light cone]{On CR maps from the sphere into the tube over the future light cone II: Higher dimensions}
\author{Michael Reiter}
\address{Fakultät für Mathematik, Universität Wien, Oskar-Morgenstern-Platz 1, 1090 Wien, Austria}
\email{m.reiter@univie.ac.at}
\author{Duong Ngoc Son}
\address{Faculty of Fundamental Sciences, PHENIKAA University, Hanoi 12116, Vietnam}
\email{son.duongngoc@phenikaa-uni.edu.vn}
\date{18 June 2024}
\begin{abstract}
We determine all CR maps from the sphere in \(\mathbb{C}^3\) into the tube over the future light cone in \(\mathbb{C}^4\). This result leads to a complete characterization of proper holomorphic maps from the three-dimensional unit ball into the classical domain of type IV of four dimension and confirms a conjecture of Reiter--Son in \cite{rs22} from 2022. Additionally, we prove a boundary characterization of isometric holomorphic embeddings from a ball into a classical domain of type IV in arbitrary dimensions that is similar to the main result in Huang--Lu--Tang--Xiao \cite{huang2020boundary}. The result is then used to treat a special case in the general characterization.
\end{abstract}
\maketitle

\section{Introduction}

Let \(\mathbb{H}^5 \subset \mathbb{C}^3\) be the Heisenberg hypersurface defined by
\[ 
\Im(w) - z \overline{z^t} = 0, \quad z = (z_1,z_2),
\]
and let $\mathcal{X}$ be a local model for the tube over the future light cone in \(\mathbb{C}^4\) given by
\[ 
\Im(w) - \frac{z \overline{z^t} + \Re(\bar{\zeta}z z^t)}{1 - |\zeta|^2}, \quad |\zeta|<1, \quad (z,\zeta,w) = (z_1,z_2,\zeta, w) \in \mathbb{C}^4.
\]
We are interested in the characterization of CR maps from \(\mathbb{H}^5\) into \(\mathcal{X}\). This problem is motivated by research on the characterization of CR maps between real hypersurfaces as well as proper holomorphic maps between domains in complex spaces. Interesting models for the CR map problem are those with ``large'' groups of CR automorphisms and their symmetry algebras (the Lie algebras of local CR automorphisms) such as spheres, tubes over the future light cone, smooth and Shilov boundaries of classical symmetric domains, and many others. For the proper holomorphic map characterization, we are interested in holomorphic maps between balls and classical symmetric domains of various types and dimensions. The related literature is huge and goes back to Henry Poincar\'e \cite{poincare1907fonctions} in 1907.  For later works, we mention for examples previous works related to CR maps of Webster \cite{Webster79b}, Faran \cite{faran1982maps,faran1986linearity}, Forstneric \cite{forstneric1989extending}, D'Angelo \cite{dangelo1988proper}, Della Sala et al \cite{della2020sufficient}, Ebenfelt \cite{ebenfelt04}, Huang \cite{Huang99}, Kim--Zaitsev \cite{kim}, Lamel \cite{lamel2019cr}, Reiter \cite{Reiter16a} and the numerous references therein. For works on proper holomorphic maps, we refer the readers to, e.g., Alexander \cite{alexander1977proper}, Chan--Mok \cite{chan2017holomorphic}, Mok \cite{Mok2011,Mok2018}, Mok--Ng \cite{mok2012germs}, Xiao--Yuan \cite{xiao2020holomorphic} and the references therein for more information. Finally, we should also mention Reiter--Son \cite{rs22} which treats the lower dimensional case, namely, the case of the sphere in \(\mathbb{C}^2\) and the tube over the future light cone in \(\mathbb{C}^3\). That paper is directly related to the present manuscript. 

Let us denote by \(\Aut(\mathbb{H}^5)\) and \(\Aut(\mathcal{X})\) the CR automorphism groups of the Heisenberg hypersurface and the local model of the tube over the future light cone, respectively. Two germs of CR maps \((H,p)\) and \((\tilde{H}, q)\) from \(\mathbb{H}^5\) into \(\mathcal{X}\) are said to be equivalent if there exist \(\psi \in \Aut(\mathbb{H}^5)\) and \(\gamma \in \Aut(\mathcal{X})\) such that \(\psi(p) = q\), \(\gamma(H(p)) = \tilde{H}(q)\) and 
\begin{equation}\label{e:germeq}
	\tilde{H} = \gamma \circ H \circ \psi^{-1}.
\end{equation}
Then, our main result can be stated as follows.
\begin{theorem}\label{main} Let \(U\) be an open subset of \(\mathbb{H}^5\) and \(H\colon U \to \mathcal{X} \subset \mathbb{C}^4\) a \(C^2\)-smooth CR map. Then the following hold:
\begin{enumerate}[(a)]
	\setlength{\itemsep}{3pt}
	\item If \(H\) is CR transversal at a point, then it is transversal on \(U\). The germs \((H,q)\), \(q\in U\), are mutually equivalent and are equivalent to exactly one of the germs at the origin of the following maps:
	
	\medskip 
	
	\begin{enumerate}[(i)]
		\setlength{\itemsep}{3pt}
		\item \(\ell(z,w) = (z,0,w)\),
		\item \(r(z,w) = \left(\dfrac{z(I+iwA)}{1 - w^2},\ \dfrac{ 2zAz^t}{1 - w^2},\ \dfrac{w}{1 - w^2} \right), \ \text{with}\ A = \begin{pmatrix} 1 & 0 \\ 0 & -1
		\end{pmatrix},\) and $I$ 
		is the $2\times 2$-identity matrix,
		\item \(\iota(z,w) = \dfrac{2}{1 + \sqrt{1-4w^2 - 4iz z^t}}\left(z,\, w, \, w \right)\).
	\end{enumerate}  
	\item \(H\) is nowhere CR transversal and equivalent to a map $(z,w) \mapsto (0,\phi(z,w),0)$ for a $C^2$-smooth CR function $\phi$ with $\phi(0)=0$.
\end{enumerate}
\end{theorem}
Here and throughout this article we write \(z = (z_1,z_2) \in \mathbb{C}^2\).
The linear map \(\ell(z,w)\)  and the irrational map \(\iota(z,w)\) are holomorphic isometric embeddings with respect to certain K\"ahler metrics on one-sided neighborhoods of the source and the target. These maps were known earlier in Reiter--Son \cite{rs22}. Precisely, two maps \(\ell(z,w)\) and \(\iota(z,w)\) can be constructed from the well-known isometric embeddings from the 3-ball into the type IV domain of four dimensions given in Xiao--Yuan \cite{xiao2020holomorphic} while the rational map \(r(z,w)\) appeared in Reiter--Son \cite{rs22}. Actually, it was conjectured in that paper that there are only three equivalence classes of CR maps represented by these maps. Thus, the present paper confirms that conjecture.

We also point out that, in contrast with the case of mapping between spheres, the rigidity for the CR codimension one maps only fails when the source is of one CR dimension \cite{Webster79b}.

As briefly mentioned above, the lower dimensional case was treated in Reiter--Son \cite{rs22}. The result was then applied to characterize the proper holomorphic maps from \(\mathbb{B}^2\) into \(D^{\mathrm{IV}}_3\subset \mathbb{C}^3\) which extend sufficiently smooth to a boundary point. In higher dimensions, proper holomorphic maps from a ball into the a classical domain of type IV in ``low codimension'' exhibit a rigidity property, as a consequence of the work of Xiao--Yuan \cite{xiao2020holomorphic}. Theorem~\ref{main} above leads to a characterization of proper holomorphic maps in the case that has been left open in these papers, namely, the case of proper holomorphic maps from a ball in \(\mathbb{C}^3\) into the classical domain of type IV in \(\mathbb{C}^4\). To be more precise, we recall that the ball in \(\mathbb{C}^3\) is given by
\[ 
\mathbb{B}^3 = \left\{(z_1,z_2,w) \in \mathbb{C}^3 \mid |w|^2 + z \overline{z^t} = 1, \ z=(z_1,z_2) \right\},
\]
and the classical domain of type IV domain is given by Cartan \cite{cartan1935domaines} (cf. Hua \cite{hua1963harmonic})
\[ 
D^{\mathrm{IV}}_4 = \left\{Z \in \mathbb{C}^4 \mid 1 - 2 Z \overline{Z}^t + \left|Z Z^t\right|^2 > 0, \ |Z|\leq 1 \right\},
\]
where \(Z^t\) is the transposition of \(Z\). If \(H\) and \(H'\) are two holomorphic maps from \(\mathbb{B}^3\) into \(D^{\mathrm{IV}}_4\), we say that \(H\) and \(\tilde{H}\) are equivalent if there exist automorphisms \(\gamma\in \Aut(\mathbb{B}^3)\) and \(\psi\in \Aut(D^{\mathrm{IV}}_4)\) such that $$ H = \gamma^{-1} \circ \tilde{H} \circ \psi.$$ Then we can state our characterization of proper holomorphic maps as follows.

\begin{corollary}\label{cor:balld4}
Let $H\colon \mathbb{B}^3 \to D^{\mathrm{IV}}_4$ be a proper holomorphic map which extends smoothly to some boundary point $p\in \partial \mathbb{B}^3$. Then \(H\) is equivalent to one of the following pairwise inequivalent maps:
\begin{enumerate}[(i)]
	\setlength{\itemsep}{3pt}
	\item 
	\begin{equation}
		R_0(z,w) = \left(\dfrac{z}{\sqrt{2}},\dfrac{2 w^2+2 w-zz^t}{4 (w+1)},\dfrac{i \left(2 w^2+2 w+zz^t\right)}{4 (w+1)}\right),
	\end{equation}
	\item 
	\begin{equation}
		I(z,w) = \left(z,\ w,\ 1-\sqrt{1-zz^t-w^2}\,\right)\bigl/ \sqrt{2},
	\end{equation}
	\item 
	\begin{equation}
		P(z_1,z_2,w) = \left(z_1,\ z_2w,\ \dfrac{w^2-z_2^2}{2},\ \dfrac{i(w^2+z_2^2)}{2}\right),
	\end{equation}
	where \(z = (z_1,z_2)\) so that \(zz^t = z_1^2 + z_2^2\).
\end{enumerate}
\end{corollary}

\begin{remark}\label{rem:1}\rm 	
The regularity assumption for the map can be reduced, see Mir \cite{Mir17}, Xiao \cite{Xiao20}, Kossovskiy--Lamel--Xiao \cite{KLX21}, Greilhuber \cite{Greilhuber20}. The maps \(R_0\) and \(I\) are isometric embeddings of the ball into the classical domain with respect to the canonical Bergman metrics, as appeared in Mok \cite{Mok2018}, Upmeier--Wang--Zhang \cite{upmeier2019holomorphic}, Xiao--Yuan \cite{xiao2020holomorphic}, while \(P\) is not isometric.

The map \(P\) was found in Reiter--Son \cite{rs22}, which is related to two quadratic polynomial maps in the lower dimensional case, and was discovered in that paper. When restricting to the complex hyperplane \(z_1 = 0\), we obtain
\[ 
P(0,z_2,w) = \left(0,z_2w, \dfrac{w^2-z_2^2}{2}, \dfrac{i(z_2^2+w^2)}{2}\right)
\]
which leads to the map sending \(\mathbb{B}^2\subset \mathbb{B}^3\) into \(D^{\mathrm{IV}}_3\) and appears in that paper. Similarly, when restricting to \(\{z_2 = 0\}\), we obtain
\[ 
P(z_1,0,w) = \left(z_1,0,\dfrac{w^2}{2}, \dfrac{iw^2}{2}\right)
\]
which is almost the same as the previous known map. Finally restricting to \(\{w=0\}\), we obtain
\[ 
P(z_1,z_2,0) = \left(z_1,\dfrac{-z_2^2}{2}, \dfrac{iz_2^2}{2}\right)
\]
which also leads to a map known in lower dimensions.
\end{remark} 

The rest of this paper can be summarized as follows. In Section~\ref{sec:2}, we introduce the related models and the stability groups. In Section~\ref{sec:3}, we give a normalization of CR maps, which is the first step in our proof. We also discuss the notion of geometric rank of the maps and its relation to the extendability to a local isometric embedding of one-sided neighborhoods of the source and target, which will shorten our proof of the main characterization (as compared to Reiter--Son \cite{rs22}). In Section~\ref{sec:4}, we give a detailed proof of the main theorem and, finally, in Section~\ref{sec:5} we prove the  corollary on the characterization of proper holomorphic maps and provide further examples and constructions.

\section{Preliminaries}\label{sec:2}
\subsection{The tube over 
	the future light cone}
	
The tube over the future light cone is the tube manifold \(\mathcal{T}: = \mathcal{C} \times i\mathbb{R}^4\), where 
\begin{equation}\label{flc}
	\mathcal{C} = \left\{(x_1, x_2, x_3, x_4) \in \mathbb{R}^4 \mid x_1^2 + x_2 ^2 + x_3^2 = x_4^2,\ x_4 > 0\right\}
\end{equation}
is the future light cone in \(\mathbb{R}^4\). The tube over the future light cone in general dimension (which can be defined completely similar) is an interesting model for Levi-degenerate hypersurfaces which appeared in various papers, see, for examples, Fels--Kaup \cite{fels2007cr}, Gregorovi\v c--Sykes \cite{gs}, and the references therein. As a real hypersurface in \(\mathbb{C}^4\), this tube has the defining function
\begin{equation}\label{tflc}
	\rho(z_1,z_2,z_3,w) := -\left(\Re w\right)^2 + \left(\Re z_1\right)^2 + \left(\Re z_2\right)^2 +\left(\Re z_3\right)^2,
\end{equation}
with \(\Re w > 0\). The holomorphic map 
\begin{equation}\label{t2m}
	F(z_1,z_2,z_3,w) = \left(\frac{2z}{1+w-z_3}, \ \frac{1-w+z_3}{1+w-z_3}, \ \frac{2i(w+w^2 - zz^t + z_3 - z_3^2)}{1+w-z_3}\right), 
\end{equation}
has a singular locus  \(\{1+ w - z_3 = 0\}\), sends \(p=(0,0,-1/2,1/2)\) to the origin \((0,0,0,0)\), and sends a neighborhood of \(p\) in \(\mathcal{T}\) into the hypersurface \(\mathcal{X}\). This shows the local CR equivalence of  \(\mathcal{T}\) and \(\mathcal{X}\).

Taking the (local) inverse of \(F\), we obtain the  map 
\begin{equation}\label{x2t}
	G(z,\zeta,w) = \left(\frac{z}{1+\zeta},\ \frac{-2 + zz^t + 2\zeta - iw(1+\zeta)}{4(1+\zeta)},\ \frac{2 + zz^t - 2\zeta - iw(1+\zeta)}{4(1+\zeta)}\right),	
\end{equation}
which sends the origin to the point \(p=(0,0,-1/2,1/2)\) and \(\mathcal{X}\) into \(\mathcal{T}\).

Using \(G\), we can easily construct CR maps from the Heisenberg hypersurface \(\mathbb{H}^5\) into the tube over the future light cone \(\mathcal{T}\). Precisely, composing this map with the linear map \(\ell(z,w) = (z,0,w)\) from \(\mathbb{H}^5\) into \(\mathcal{X}\), we obtain the following quadratic map
\begin{equation}\label{h2t0}
	T_1(z,w) = (G\circ \ell )(z,w) =  \left(z,\ \frac{1}{4} (zz^t -iw-2),\ \frac{1}{4}(zz^t-iw+2)\right),
\end{equation}
which sends the origin to \(p=(0,0,-1/2,1/2)\) and sends \(\mathbb{H}^5\) into \(\mathcal{T}\).

Composing \(G\) with the rational map \(r(z,w)\), we obtain
\begin{align}\label{h2t1}
T_2& (z,w) 
= (G\circ r)(z,w)  \notag \\
&= \left(\frac{z(I + iwA)}{1 - w^2 + 2zAz^t},\ \frac{
	2w^2 - iw -2 + z(I +4A)z^t}{4(1 - w^2 + 2zAz^t)},\ \ \frac{
	2w^2 + iw + 2 + z(I - 4A)z^t}{4(1 - w^2 + 2zAz^t)}\right),
\end{align}
where, as before, \(A = \mathrm{diag}(1,-1) \in \mathrm{Mat}(2,\mathbb{R})\).

Finally, by composing \(G\) with the irrational map, we obtain an irrational map from \(\mathbb{H}^5\) into \(\mathcal{X}\). However, the formula for this map is quite complicated and we refrain from presenting the details here.

These three maps represent all equivalence classes of CR maps from the Heisenberg hypersurface into the tube over the future light cone.

\subsection{Stability groups}
The stability groups (at the origin) \(\Aut_0(\mathbb{H}^5)\) of the Heisenberg hypersurface \(\mathbb{H}^5\) and \(\Aut_0(\mathcal{X})\) of the local model \(\mathcal{X}\) are important for us. We use them to normalize CR maps from \(\mathbb{H}^5\) into \(\mathcal{X}\). \(\Aut_0(\mathbb{H}^5)\) has a well-known and simple parametrization as follows:  Let $s>0, u\in \CC, |u|=1, c=(c_1,c_2)\in \CC^2, r\in \RR, U \in U(2)$, 
\begin{align*}
U = \left(\begin{array}{cc}
u a & - u b\\
\bar b & \bar a
\end{array} \right), 
\end{align*}
with $a,b \in \CC$ satisfying $|a|^2 + |b|^2 = 1$ and
\begin{align*}
\delta =  1 - 2i z \bar c^t  + (r- i c \bar c^t)w,
\end{align*}
such that the stability group $\Aut_0(\mathbb{H}^5)$ is given by the following automorphisms:
\begin{align*}
(z,w) \mapsto \psi_{s,u,U,c,r}(z,w) = \left(s (z + c w) U,  s^2 w \right)/\delta.
\end{align*}
The group \(\Aut_0(\mathcal{X})\)  can be computed by integrating the vector fields in its symmetry algebra which vanish at the origin. To make our formulas more concise, we put
\[ 
	\delta = \delta(z,w) = 1 - (r'+i a \bar{a}^t) w - 2i z \bar{a}^t+ i\, \overline{aa^t} (w\zeta + i z z^t),
 \]
 for \(a = (a_1, a_2) \in \mathbb{C}^2\), \(u'\in \mathbb{C}\), \(|u'| = 1\), \(r' \in \mathbb{R}\), and \(P \in \mathrm{O}(2)\). The stability group consists of holomorphic maps of the form \(\gamma = (\gamma_1, \dots, \gamma_4)\), where
\begin{align}
\eta(z,w) & = s' u' (z+ w a - (w\zeta + i z z^t)\bar{a})P/\delta, \quad \eta = (\gamma_1,\gamma_2) \\
\gamma_3(z,w) & = {u'}^2\left(\zeta -2 z a^t -i aa^t w  -(r'-i \bar{a}a^t )(w\zeta + i z z^t)\right)/\delta, \\
\gamma_4(z,w) & = {s'}^2 w/\delta,
\end{align}
where $s'>0$.
This form is completely similar to the case of the tube over the future light cone in \(\mathbb{C}^3\).
We parametrize elements of \(\Aut_0(\mathcal{X})\) by
\begin{align*}
(z,w) \mapsto \psi'_{s',u',P,a,r'}(z,w) = \gamma(z,w).
\end{align*}

\section{Normalization, 
	geometric rank, and isometric embeddings}\label{sec:3}
	
In this section, we explain the first step in our proof of the main result, namely, the normalization process. This is used as starting point in order to introduce the notion of geometric rank for a map. This is similar to Reiter--Son \cite{rs22} and is ultimately motivated by Huang \cite{Huang99}. Based on an idea originated in Lamel--Son \cite{lamel2019cr} and Reiter--Son \cite{rs22}, we introduce a tensorial invariant for CR transversal maps from a sphere or hyperquadric into the boundary of a classical domain of type IV and prove a version of Huang--Lu--Tang--Xiao \cite{huang2020boundary} boundary characterization of isometric embeddings.

\subsection{Normalization}
Write $H = (f,\phi,g)=(f_1, f_2, \phi, g) \colon \mathbb{C}^3 \to \mathbb{C}^4$ for a holomorphic map sending the Heisenberg hypersurface $\Im w - z \overline{z^t} = 0$ into the model \(\mathcal{X}\) defined by
\[
(1- |\zeta|^2) \Im w - z \overline{z^t} - \Re \left(\bar \zeta z z^t \right) = 0
\]
and maps $(0,0,0)$ to $(0,0,0,0)$. 
As \(H(U\cap\mathbb{H}^5)\subset \mathcal{X}\), the following identity holds
\begin{multline}\label{e:me}
\bigl(1-\phi(z,\wba+2i z \zba^t) \bar{\phi}(\zba,\wba)\bigr) \bigl(g(z,\wba+2i z \zba^t) - \gba(\zba,\wba)\bigr)- f(z,\wba+2i z \zba^t) \bar f^t(\zba,\wba)\\ - \frac{1}{2}\left\{\bar{\phi}(\zba,\wba)F(z,\wba+2i z \zba^t) + \phi(z,\wba+2i z \zba^t)\bar{F}(\zba,\wba)  \right\} = 0,
\end{multline}
where \(F = f f^t\). We shall call \eqref{e:me} the \emph{mapping equation}.

If \(H\) is a germ at the origin of smooth CR maps from the Heisenberg hypersurface to \(\mathcal{X}\), then by a ``standard'' construction, we can associate \(H\) to a formal holomorphic map, still denoted by \(H \colon (\mathbb{C}^3, 0) \to (\mathbb{C}^4, 0)\) (i.e., the components \(f_1, f_2, \phi, g\) of \(H\) are formal power series in \((z,w)\) with no constant term) which sends the germ at the origin of the Heisenberg hypersurface into the model \(\mathcal{X}\) in formal sense (that is, the mapping equation holds in the ring of formal power series \(\mathbb{C}[[z,w,\bar{z},\bar{w}]])\), we refer to Baouendi--Ebenfelt--Rothschild \cite{ber} for the details of this construction. Thus, in what follows, we shall view \eqref{e:me} as an equation in \(\mathbb{C}[[z,w,\bar{z},\bar{w}]]\).

Using the stability groups of the source and the target, we can bring the map into the following partial normal form. 

\begin{lemma}
Let $p \in \mathbb{H}^5$ and $H=(f,\phi,g)$ be a germ at $p$ of a smooth transversal CR map which sends $\mathbb{H}^5$ into $\mathcal{X}$. Then the germ $(H,p)$ is equivalent to the germ at the origin of a CR map $\widetilde H=(\widetilde f,\widetilde \phi, \widetilde g)$ which is of the following form:
\begin{align*}
\widetilde f(z,w) & = z + \frac{i}{2} w(z\widetilde A) + \nu w^2  + O(3), \\
\widetilde \phi(z,w) & = \lambda w + z\widetilde B z^t + w z \mu^{t} + \sigma w^2 + O(3),\\ 
\widetilde g(z,w) & = w + O(3),
\end{align*}
where $\widetilde A, \widetilde B \in \mathrm{Mat}(2\times 2; \RR)$, $\nu=(\nu_1,\nu_2) \in \CC^2$ and $\lambda, \sigma \in \CC$. The entries of the matrices we denote by $\alpha_{ij}$ and $\beta_{ij}$ respectively and it holds that $\widetilde B$ is symmetric.

The same holds for a transversal formal map \(H\) at the origin sending \(\mathbb{H}^5\) into \(\mathcal{X}\).
\end{lemma}

\begin{proof}
Without loss of generality we can assume that $H$ sends the origin to the origin. Moreover, we can view \(H\) as a formal holomorphic map which satisfies the mapping equation \eqref{e:me}. 

We write $E_m := f_{z_m}(0)$ for $m=1,2$. From the mapping equation it follows that $g(z,0) = 0$, $g_w(0) = \|E_1\|^2 = \|E_2\|^2$ and $E_1 {\overline{E_2}}^t = 0$. The transversality of $H$ implies $g_w(0)>0$ such that $E_1$ and $E_2$ are linearly independent. We apply automorphisms successively and  write $H_{[k+1]} := \varphi_k' \circ H_{[k]} \circ \varphi_k$, where $H_{[0]} := H$, $\varphi_k'$ and $\varphi_k$ are automorphisms of the source and target respectively, parametrized as above and $k \in \NN$.
We denote by $E\in \mathrm{Mat}(2\times 2; \CC)$ the matrix whose $j^{th}$ column is $E_j$. Then we have  
\begin{align*}
H_{[1] z}(0) = \left(u u' s s'  E U,{u'}^2 s (\phi_z(0) - 2 i a' E ) U,0 \right).
\end{align*}
Then we choose $s, U$ and $a$ such that $H_{[1]z}(0) = (I,0)$, where $I$ is the $2 \times 2$-identity matrix. Considering $H_{[2]} = \varphi_2' \circ H_{[1]} \circ \varphi_2$ with $s = 1/s', a' = 0, U = 1/u' I$, we obtain $g_{[1]w}(0) = 1$ and we have
\begin{align*}
H_{[2] w}(0) = \left(c + u' f_{[1]w}(0)/s', {u'}^2 \phi_{[1]w}(0)/{s'}^2,1 \right).
\end{align*}
Choosing $c$ accordingly, this allows us to assume $f_{[2] w} = 0$. This implies, by considering the mapping equation, that $g_{[2] zw }(0)= 0$ and all $z_1$ and $z_2$ derivatives of $f_{[2]}$ of order $2$ are $0$. Finally, since by the mapping equation it holds that $g_{[2]w^2}(0) \in \RR$, we can choose $r$ in order to assume $g_{[2] w^2}(0)=0$.
\end{proof}

Applying \(\partial_{z_1}^2\partial_{\zba_1}\partial_{\zba_2}\) to the mapping equation and evaluating the result at the origin, we obtain
\[ -\bar{\phi}^{(1,1,0)}-4 i f_2^{(1,0,1)} = 0. \]
Similarly,
\[-\bar{\phi}^{(1,1,0)}-4 i f_1^{(0,1,1)} = 0. \]
We deduce from these two equations that
\[ \overline{\beta_{12}} = 2\alpha_{12} = 2\alpha_{21}. \]

Applying \(\partial_{z_1}^2\partial_{\zba_2}^2\) to the mapping equation and evaluating the result at the origin, we obtain
\[ 
\phi ^{(2,0,0)} + \bar{\phi} ^{(0,2,0)} = 0.
\]
We find that 
\[ 
\beta_{11} = - \beta_{22}
\]
and both are real.

Next, applying \(\partial_{z_1}^2\partial_{\zba_1}^2\) to the mapping equation and evaluating the result at the origin, we obtain
\[ \alpha_{11} = \beta_{11}. \]
Similarly, we also find that
\[ 
\alpha_{22} = \beta_{22}.
\]
Next, applying \(\partial_{z_1}\partial_{z_2}\partial_{\zba_2}^2\) to the mapping equation and evaluating the result at the origin, we obtain
\[ \beta_{12} = 2\alpha_{21}.\]

In summary, we can rewrite the map as follows.
\begin{proposition}\label{prop:2.1} Let \(U\) be an open neighborhood of a point \(p\) in \(\mathbb{C}^3\), \(p\in \mathbb{H}^5\), and let \(H \colon U \to \mathbb{C}^4\) be a holomorphic map such that \(H(U \cap \mathbb{H}^5) \subset \mathcal{X}\). Then there exist automorphisms \(\phi\) and \(\psi\) of \(\mathbb{H}^5\) and \(\mathcal{X}\), respectively, which satisfy \(\psi(p) = 0\), \(\gamma(H(p)) = 0\), and
$$\gamma \circ H \circ \psi^{-1} = (f,\phi, g),$$
where \(f, \phi\), and \(g\) take the following form
\begin{equation}\label{e:normalform}
	\begin{cases}
		f = z + \frac{i}{2}w(zA_{\alpha,\beta}) + w^2 \nu + O(3), \\
		\phi = \lambda w + zA_{\alpha,\beta} z^t + w z \mu^t+ \sigma w^2 + O(3),\\ 
		g = w + O(3),
	\end{cases}
\end{equation}
where
\[
A_{\alpha,\beta} = \begin{pmatrix}\alpha & \beta \\ \beta & -\alpha \end{pmatrix} \in \mathrm{Mat(2\times 2; \mathbb{R})}, \quad \nu = (\nu_1,\nu_2) \in \CC^2, \quad \mu = (\mu_1,\mu_2) \in \CC^2.\]
Moreover, the rank of \(A_{\alpha,\beta}\) does not depend on the pair \((\gamma,\psi)\) satisfying the conditions above.
\end{proposition}

\subsection{Geometric rank and the CR Ahlfors tensor}
By Proposition~\ref{prop:2.1}, the rank of the matrix \(A\) appearing in the partial normal form \eqref{e:normalform} is an invariant of the equivalence class of the germ \((H,p)\). Similarly to Huang \cite{Huang99} and Reiter--Son \cite{rs22}, we make the following definition.
\begin{definition}\label{defn:grank}
The rank of the matrix \(A\) is called the  \emph{geometric rank} of \(H\) at \(p\), and denoted by \(\rk(H)(p)\). 
\end{definition}

The notion of geometric rank is very useful in the study of sphere and hyperquadric maps, as exhibited in Huang \cite{Huang99} and Huang et al \cite{huang2020boundary}. Moreover, it is related to the rank of the Hermitian part of the CR Ahlfors tensor of sphere and hyperquadric maps, as shown in Lamel--Son \cite{lamel2019cr} and Reiter--Son \cite{rs24}. Motivated by these works, we introduce the following tensor.

\begin{definition}\label{defn:21} Let $M$ and \(M'\) be real hypersurfaces in $\mathbb{C}^{n+1}$ and \(\mathbb{C}^{N+1}\), defined by $\rho$ and \(\rho'\), respectively. Suppose that $H  \colon U \to \mathbb{C}^{N+1}$ is a holomorphic map such that \(H(U\cap M) \subset M'\). Assume that \(V \subset U\) is an open subset with \(V\cap M \ne \emptyset\) and \(Q \colon V \to \mathbb{R}\) is a positive real-valued function satisfying 
\begin{equation}\label{e:genme}
	\rho'(H(z),\overline{H(z)}) = Q(z,\bar{z})\, \rho(z,\bar{z}), \quad z\in V \subset U.
\end{equation}
Then we define a tensor \(\mathcal{A}'(H)\) associated to $H$ on \(V \cap M\) as follows:
\begin{equation}\label{gAhlfors}
	\mathcal{A}'(H)(Z, \overline{W}) = (i\partial \bar{\partial} \log Q)(Z,\overline{W}), \quad Z,W \in T^{(1,0)}M.
\end{equation}
\end{definition}
Observe that this tensor depends on the map \(H\) as well as the defining functions \(\rho\) and~\(\rho'\). The motivation for this definition comes from two sources. The first one is a recent study of the CR Ahlfors tensor and the second is the relation between the Hermitian part of the CR Ahlfors tensor and the geometric rank of sphere and hyperquadric maps.
\begin{proposition}[Lamel--Son \cite{lamel2019cr}]
Let \(M\) and \(M'\) be strictly pseudoconvex real hypersurfaces as in Definition~\ref{defn:21}. Consider the pseudo-Hermitian structures \(\theta = i\bar{\partial} \rho\) and \(\theta'= i\bar{\partial} \rho'\), respectively. 
Let \(\mathcal{A}(H)\) be the CR Ahlfors tensor of \(H\) with respect to this pair of structures. 
Assume that \(J(\rho) = 1 + O(\rho^2)\) and \(J(\rho') = 1 + O(\rho'^2)\).
Then
\[
\mathcal{A}(H) (Z_{\alpha}, Z_{\bba}) =  \mathcal{A}'(H) (Z_{\alpha}, Z_{\bba}).
\]
\end{proposition}

The construction of the CR Ahlfors tensor is lengthy. We refer the readers to \cite{lamel2019cr} for the details as we mainly work with \(\mathcal{A}'(H)\), which is defined also in the case of Levi-degenerate hypersurfaces. On the other hand, this proposition suggests that the CR Ahlfors derivative is interesting only when we can choose ``nice'' pseudohermitian structures, which often come from a special defining functions of the hypersurfaces.

Fix a local frame \(Z_{\alpha}\) in a neighborhood of \(p\), we obtain a Hermitian \(n\times n\)-matrix \((\mathcal{A}(H)(Z_{\alpha}, Z_{\bba}))\), here \(n\) is the CR dimension of the source. It is clear that the rank of the matrix on the left does not depend on the chosen frame. Moreover, as already observed by Lamel and Son in the case of sphere maps (see also Reiter--Son \cite{rs22,rs24}), we have the following lemma whose proof is left to the readers.
\begin{lemma} Let \(U\) be an open neighborhood of a point \(p\) in \(\mathbb{C}^3\), \(p\in \mathbb{H}^5\), and let \(H \colon U \to \mathbb{C}^4\) be a holomorphic map such that \(H(U \cap \mathbb{H}^5) \subset \mathcal{X}\). Then \(\rk(H)(p) = \rk \mathcal{A}'(H)(p)\). 
\end{lemma}

This lemma provides a simple argument of the invariant property of the geometric rank of \(H\).
\subsection{Isometric embeddings}
Assume that \(H\) is defined  and holomorphic in a neighborbood \(U\) of a point \(p \in M\), transversal to \(M'\) at \(H(p)\) and sends \(M\) into \(M'\). Also assume that the relation \eqref{e:genme} is valid in \(U\). In this situation, the tensor \(\mathcal{A}'(H)\) is closely related to the isometry property with respect to suitable K\"ahler metrics on one-sided neighborhoods of \(M\) and \(M'\). For example, let the sphere \(\mathbb{S}^{2n+1}\) be defined by \(\rho(z,\zba)=|z|^2 -1 = 0\) and let \(\mathcal{R}\) be the smooth boundary part of a type IV domain, defined by \(\rho'(Z,\Zba) = 1 - 2 |Z|^2+|Z Z^t|^2 = 0\).
A holomorphic map \(H\) defined in an open neighborhood \(U\) of a point \(p\in \mathbb{S}^{2n+1}\) and sends \(\mathbb{S}^{2n+1}\) into \(\mathcal{R}\) must satisfy
\begin{equation}\label{e:me2}
	\rho' \circ H = Q\rho,
\end{equation}
for some real function \(Q\). If \(H\) is transversal to \(\mathcal{R}\), then \(Q\) is positive along \(U\cap \mathbb{S}^{2n+1}\). If \(i\partial \bar{\partial} \log \rho\) and \(i\partial \bar{\partial} \log \rho'\) define K\"ahler metrics on open sets \(U \cap \{\rho>0\}\) and \(V\cap \{\rho'>0\} \supset H(U)\) respectively, then the isometry property 
of \(H\) is equivalent to the pluriharmonicity of the function \(u = \log (Q)\). In fact, \eqref{e:me2} implies that
\begin{equation}\label{e:iso}
H^{\ast}\left(i\partial \bar{\partial} \log \rho'\right) = i\partial \bar{\partial} \log \rho + i\partial \bar{\partial} \log Q.
 \end{equation}
Thus, \(H\) is an isometry if and only if \(i\partial \bar{\partial} \log Q = 0\), which in turn is equivalent to the pluriharmonicity of \(\log Q\). 
Therefore, the isometry property of \(H\) implies the vanishing of the tensor \(\mathcal{A}'(H)\) on \(U \cap M\). This simple observation is particularly interesting when considering classical domains when the boundary defining function is related to the Bergman kernels of the domains. More explicitly, the Bergman kernel of the classical symmetric domain of type IV is 
\[ 
K_{D^{\mathrm{IV}}_m}(Z,Z') = \frac{1}{V(D^{\mathrm{IV}}_m)} \left(1-2\overline{Z} Z' + |ZZ'|^2\right)^{-m},
\]
therefore, the K\"ahler form of the Bergman metric on this domain is given by
\[ 
\omega_{D^{\mathrm{IV}}_m} = i\partial\bar{\partial} \log K_{D^{\mathrm{IV}}_m}(Z,Z)
=
-m(i\partial\bar{\partial} \log \rho'(Z,\Zba))
\]
Likewise, the K\"ahler form for the Bergman metric on the ball is
\[ 
\omega_{\mathbb{B}^{n+1}} = i\partial\bar{\partial} \log K_{\mathbb{B}^{n+1}}(Z,Z)
=
-(n+2)i\partial\bar{\partial} \log \rho(z,\zba).
\]
Thus, the restriction of \(H\) to \(U\cap \mathbb{B}^{n+1}\) is an isometric embedding (up to a normalizing constant), which means that $H$ satisfies the following equation 
\[ 
\frac{m}{n+2}\omega_{\mathbb{B}^{n+1}} = H^{\ast} \omega_{D^{\mathrm{IV}}_m},
\]
which holds if and only if \(H\) is a local embedding and \(Q\) is pluriharmonic on \(U\cap \mathbb{B}^{n+1}\). By continuity, \(\mathcal{A}'(H) = 0\) on \(U\cap \mathbb{S}^{2n+1}\), i.e., \(H\) has vanishing geometric rank on \(U\cap \mathbb{S}^{2n+1}\). The converse also holds. In fact, we have the following version of Huang--Lu--Tang--Xiao \cite{huang2020boundary} boundary characterization of isometric embeddings.
\begin{theorem}\label{thm:22} Let \(U\) be an open neigborhood of a point \(p \in \mathbb{S}^{2n+1}\) and \(H\) a holomorphic map from \(U\) into \(\mathbb{C}^m\). Assume that \(U \cap \mathbb{B}^{n+1}\) is connected, \(H(U \cap \mathbb{B}^{n+1}) \subset \tfourM\), and \(H(U \cap \mathbb{S}^{2n+1}) \subset \mathcal{R}\). Then the following are equivalent:
\begin{enumerate}
	\item \(H\) is transversal at \(p\) and \(\mathcal{A}(H) = 0\) on an open neigborhood of \(p\) in \(\mathbb{S}^{2n+1}\).
	\item \(H\) is an isometric embedding from \(U \cap \mathbb{B}^{n+1}\) into \(\tfourM\).
\end{enumerate}
\end{theorem}
In Theorem~\ref{thm:2.3} below, we also have a similar characterization of isometric embeddings for holomorphic maps sending a piece of a Heisenberg hypersurface into the model \(\mathcal{X}\). The proofs of these two theorems are essentially the same. We thus omit the proof of Theorem~\ref{thm:22}. We shall use this result in the proof of Theorem~\ref{main}.

The Siegel domain $\Omega$ is the domain in \(\mathbb{C}^3\) defined by \(\rho > 0\),
where
\[ 
\rho = \Im w - z \overline{z}^t,
\]
equipped with  the K\"ahler metric
\[ 
\omega_{\Omega} = i\partial\bar{\partial} \log(\rho(z,\zba).
\]
The domain \(\Omega\) is an unbounded model for the complex hyperbolic space with boundary \(\mathbb{H}^{2n+1}\). Similarly, the open set \(\mathcal{S} \subset \mathbb{C}^m\), \(m\geq 2\), defined by \(\rho' > 0\), where
\[ 
\rho' = (1-|\zeta|^2)\Im w  - Z \overline{Z}^t -\Re (\zeta Z Z^t),
\]
is a one-sided neighborhood of \(\mathcal{X}\). On \(\mathcal{S}\), we consider the K\"ahler metric with the K\"ahler form
\[
\omega_{\mathcal{S}} = i\partial\bar{\partial} \log(\rho'(Z,\zeta,w, \Zba,\bar{\zeta},\wba)).
\]
We have
\[ 
	\det\begin{pmatrix} \rho_{j\kba}
	\end{pmatrix}
	=	\frac{1}{4} \rho'^{-m},
 \]
which implies that
\[ 
	i\partial\bar{\partial} \log \left(\det\begin{pmatrix} \rho_{j\kba}
	\end{pmatrix}\right) = - m\, \omega_S.
 \]
That is, \(\omega_S\) is a K\"ahler--Einstein metric with scalar curvature \(R = -m^2\).

It is immediate to check that the maps \(\ell\) and \(\iota\) extend to local holomorphic embeddings from the Siegel domain \(\Omega\) into \(\mathcal{S}\). Indeed, computations yield
\[ 
	\rho' \circ \ell = \rho,
 \]
which implies that \(\ell\) is an isometric embedding, while the identity
\[ 
	\rho' \circ \iota = \left|\frac{2}{1+\sqrt{1-4w^2-4i z z^t}}\right| \rho
 \]
implies that \(\iota\) is also a local isometric embedding. Here, we use the simple fact that the logarithm of the modulus of a nonvanishing holomorphic function is pluriharmonic.

Thus, the tensor \(\mathcal{A}(\ell) = \mathcal{A}(\iota) = 0\) on \(U\cap \mathbb{H}^{2n+1}\). Similar to the theorem above, we have  \begin{theorem}\label{thm:2.3}
Let \(U\) be an open neigborhood of a point \(p \in \mathbb{H}^{2n+1}\) and \(H\) a holomorphic map from \(U\) into \(\mathbb{C}^m\). Assume that \(U \cap \Omega\) is connected, \(H(U \cap \Omega) \subset \mathcal{S}\), and \(H(U \cap \mathbb{H}^{2n+1}) \subset \mathcal{X}\). Then the following are equivalent:
\begin{enumerate}
	\item \(H\) is transversal at \(p\) and \(\mathcal{A}(H) = 0\) on an open neigborhood of \(p\) in \(\mathbb{H}^{2n+1}\).
	\item \(H\) is an isometric embedding from \(U \cap \Omega^{n+1}\) into \(\mathcal{S}\).
\end{enumerate}
\end{theorem}
\begin{proof} (1) \(\Rightarrow\) (2) follows from the discussion prior to the theorem. We leave the details to the readers. 

(2) \(\Rightarrow\) (1): Let \(\mathcal{X}\subset \mathbb{C}^{m+1}\) be the real hypersurface given by 
\[
\mathcal{X}:=\{\rho'(Z,\zeta,w):=(1- |\zeta|^2) \Im w - Z \Zba^t - \Re \left(\zeta Z Z^t \right) = 0,\ |\zeta| <1\},
\]
where $Z = (z_1,z_2,\dots, z_{m-1}) \in \mathbb{C}^{m-1}$ is a row vector and $\zeta, w \in \mathbb{C}$. The holomorphic map $\Psi \colon \mathbb{C}^{m+1} \to \mathbb{C}^{m+2}$ given by
\[
\Psi(Z ,\zeta, w)
=
\left( Z, \frac{1}{2}\left(w \zeta + i Z Z^t + i \zeta\right),  \frac{1}{2}\left(w \zeta + i Z Z^t - i \zeta\right) , w \right)
\]
transversal to \(\mathcal{X}\) and sends $\mathcal{X}$ into the indefinite real hyperquadric $ \mathbb{H}^{2m+3}_1$. In fact, if $ \mathbb{H}^{2m+3}_1$ is defined by
\[
\widetilde{\rho} (Y,\overline{Y}): = \Im y_{m+2} - \sum_{j=1}^{m} |y_j|^2 + |y_{m+1}|^2, \quad Y = (y_1, \dots , y_{m+2}) \subset \mathbb{C}^{m+2},
\]
then it holds that
\[
\widetilde{\rho} \circ \Psi = \rho'.
\]
Thus, the composition \(\Psi \circ H\) is transversal to \(\mathbb{H}^{2m+3}_1\) and sends the \(\mathbb{H}^{2n+1}\) into \(\mathbb{H}^{2m+3}_1\).

If \(\mathcal{A}(H) = 0\),  then \(\Psi \circ H\) has vanishing Hermitian part of the CR Ahlfors tensor (cf. Reiter--Son \cite{rs24}). We can apply the result of Huang--Lu--Tang--Xiao \cite{huang2020boundary} to conclude that \(\Psi \circ H\) extends as a local isometry. From this, we can prove that \(H\) also extends as a local isometry.
\end{proof}

\begin{remark}\rm Theorem~\ref{thm:2.3} can be easily generalized to the case of CR maps from a hyperquadric into the hypersurface
\[
\mathcal{X}_{l} = \left\{(1-|\zeta|^2)\Im(w) - Z E_l \Zba^t - \Re(\zeta Z Z^t) = 0, \ |\zeta|<1\right\},
\]
where \(E_l = \mathrm{diag}(1,\dots,1,\dots, -1,\dots, -1)\), with eigenvalue \(-1\) of multiplicity \(l\) . Similarly, Theorem~\ref{thm:22} can also be generalized to the case of CR maps from the boundary of a generalized ball into a generalized domain of type IV, defined by
\[ 
D_{m,l}^{\mathrm{IV}} = \left\{1 - 2 Z E_l \Zba^t + |Z Z^t|^2 > 0,\ |ZZ^t|<1\right\}.
\]
Details are left to the readers.
\end{remark}

\section{Proof of Theorem~\ref{main}}
\label{sec:4}
In this section, we prove our main theorem. The idea is similar to the proofs in Reiter--Son \cite{rs22} and Reiter \cite{Reiter16a}, however, since we have more components of the map as well as more variables, the required computations are much more challenging.

The proof consists of several steps involving the determination of the map under consideration along the first and second Segre variates of the Heisenberg hypersurface. Based on a partial normal form of the map, the first step is to determine the map along the first Segre set \(\Sigma = \{(z,w) \in \mathbb{C}^3 \mid w = 0\}\). By a ``reflection principle'' we obtain three holomorphic equations for the components of the map. In the second step, we determine the derivative \(H_{w}\) along the first Segre set, giving another holomorphic equation, which in turn completes a system of four holomorphic equations for four components of the map. In the final step, we solve the system and show that the solution is equivalent to one of the maps in the given list. Below we shall present the details.

Let \(U \subset \mathbb{H}^5\) be an open subset of the Heisenberg hypersurface and \(H\colon U \to \mathcal{X} \subset \mathbb{C}^4\) a \(C^2\)-smooth. By the regularity result mention in Remark \ref{rem:1}, \(H\) is smooth.
As both \(\mathbb{H}^5\) and \(\mathcal{X}\) are homogeneous, we can assume that \(0 \in U\) and \(H\) sends the origin into the origin: \(H(0) = 0\). By the normalization, \(H\) is equivalent to a (smooth or formal) map of the form \eqref{e:normalform}. We therefore assume that \(H\) is a formal map and already of this form.

Next, we determine \(H\) along the Segre set \(\Sigma\) of the Heisenberg hypersurface at the origin. Here,
\[ \Sigma = \{(z,0) \mid z \in \mathbb{C}^2\} \subset \mathbb{C}^3.\]
\begin{lemma}\label{lem:31}  On the first Segre set \(\Sigma\), it holds that 
\begin{equation}\label{e:lem31}
	f = \frac{2z}{1+\sqrt{1-4i\bar{\lambda}z z^t}},\quad g = 0.
\end{equation}
\end{lemma}
\begin{proof} Setting \(\zba = (0,0)\) and \(\wba=0\) in \eqref{e:me} and using the fact \(H\) sends the origin into the origin, we find that
\[ 
\bar{g}(\zba,0) = 0.
\]
Taking complex conjugation, we have \(g(z,0) = 0\), as desired.

Next, we introduce the following differential operators
\begin{equation}\label{vfs}
	L_1 = \frac{\partial}{\partial \zba_1} - 2i z_1 \frac{\partial}{\partial \wba}, \quad L_2 = \frac{\partial}{\partial \zba_2} - 2i z_2 \frac{\partial}{\partial \wba}.
\end{equation}
Applying \(L_1\) to the mapping equation \eqref{e:me} and evaluating along \((\zba,\wba) = (0,0,0)\), we obtain
\begin{equation}\label{e33}
	z_1 - f_1(z,0) + i\lba z_1 F(z,0) = 0,
\end{equation}
where, as above, \(F = ff^t = f_1^2 + f_2^2\). Similarly, we have 
\begin{equation}\label{e34}
	z_2 - f_2(z,0) + i\lba z_2 F(z,0) = 0.
\end{equation}
From \eqref{e33} and \eqref{e34}, we find that
\[ 
z_1 f_2(z,0) = z_2 f_1(z,0),
\]
which, in turns, gives
\[ 
z_1^2 F(z,0) = z z^tf_1^2(z,0) 
\]
From this equation and \eqref{e33}, we obtain a quadratic equation in \(f_1(z,0)\), namely,
\begin{equation}\label{e:35}
	z_1^2 - z_1 f_1(z,0) + i\lba z z^t f_1^2(z,0)=0.
\end{equation}
Solving \eqref{e:35} for holomorphic solution, we find that
\[ 
f_1(z,0) = \frac{2z_1}{1+\sqrt{1-4i\bar{\lambda}z z^t}},
\]
where \(\sqrt{1-4i\bar{\lambda}z z^t}\) is the branch of the root taking value 1 at \(z=(0,0)\). The formula for \(f_2(z,0)\) follows immediately.
\end{proof}
\begin{lemma}\label{lem:32} On the first Segre set \(\Sigma\), it holds that 
\[ g_{w} = 1 + i \bar{\lambda} f f^t = 1+ \frac{4i \lba z z^t}{1+\sqrt{1-4i\bar{\lambda}z z^t}}.
\]
\end{lemma}
\begin{proof} Differentiating the mapping equation with respect to \(\wba\) and evaluating along \((\zba,\wba) = (0,0,0)\), we obtain
\[ 
g_w = 1+i \bar{\lambda} f f^t + \lba g\phi.
\]
From this and the previous lemma, we complete the proof.
\end{proof}

We divide into two cases.

\subsection{Case 1:} \(\lambda = 0\) In this case, we shall prove that \(H\) is equivalent to one of three rational maps.

From Lemmas \ref{lem:31} and \ref{lem:32}, we have  
\[ 
f(z,0) = z, \quad g(z,0) = 0,\quad  g_w(z,0) = 1.
\]
\begin{lemma} If \(\lambda = 0\), then 
\begin{equation}\label{e:lem33}
    \sigma =  0, \quad \nu =\mu = (0,0).
\end{equation}
\end{lemma}
\begin{proof}
Applying \(L_1\) and \(L_2\) to the mapping equation \eqref{e:me} consecutively, then setting \(\zba_1=\zba_2 = \wba = 0\), we find that
\begin{align*}
	\alpha (z_2f_1 - z_1 f_2) + 8 z_1 z_2\, \bar{\nu} f^t + (i( \bar{\mu}_2z_1+\bar{\mu}_1z_2) + 4 \bar{\sigma} z_1z_2) F
	= 0
\end{align*}
holds along \(w=0\). Using Lemmas \eqref{lem:31} and \eqref{lem:32}, we obtain
\[ 
4 \bar{\sigma} (z_1^3z_2+z_1z_2^3) + i\bar{\mu}_2z_1^3+i\bar{\mu}_1z_2^3 + (8\bar{\nu}_1+i\bar{\mu}_1) z_1^2 z_2 + (8\bar{\nu}_2+i\bar{\mu}_2)z_1z_2^2 = 0.
\]
Equating the coefficients of the monomials on the left-hand side to zero, we can easily obtain \eqref{e:lem33}.
\end{proof}

Next, we determine \(\phi\) along \(\Sigma\).
\begin{lemma}   If \(\lambda = 0\), then
\begin{equation}\label{H1stSegre}
	H(z,0) = (z, zA_{\alpha,\beta} z^t, 0), \quad \text{where}\ A_{\alpha,\beta} = \begin{pmatrix} \alpha & \beta \\ \beta  & -\alpha 
	\end{pmatrix}.
\end{equation}
\end{lemma}
\begin{proof} It remains to prove that 
\[ 
\phi(z,0) = \alpha(z_1^2 - z_2^2) + 2\beta z_1z_2 = z A_{\alpha,\beta} z^t.
\]
To this end, we apply \(L_1^2\) to the mapping equation and set \((\zba,\wba)= (0,0,0)\). Then, solving the resulting equation for \(\phi(z,0)\) to get the desired formula.
\end{proof}
From Lemma \ref{H1stSegre}, we obtain several holomorphic equations for components of \(H\).
\begin{lemma}\label{lem35} If \(\lambda = 0\), then the following holds in a neighborhood of the origin of \(\mathbb{C}^3\).
\begin{align}
	\alpha w^2 (g\phi+iff^t) +4 w z_2 f_2 - iw^2 \phi + 4z_2^2 g & = 0, \label{h1}\\
	\alpha w^2 (g\phi+iff^t) - 4 w z_1 f_1 + i w^2 \phi + 4z_1^2 g & = 0, \label{h2}\\
	w(\alpha z_1 + \beta z_2) (g\phi+iff^t) - 2 z_2 (z_2 f_1 - z_1 f_2) - i w z_1 \phi & = 0,\label{h3}\\
	w(\beta z_1 - \alpha z_2) (g\phi+iff^t) + 2 z_1 (z_2 f_1 - z_1 f_2) - i w z_2 \phi & = 0,\label{h4}\\
	(zA_{\alpha,\beta}z^t) (g\phi + i f f^t) - i z z^t \phi & = 0. \label{h5}
\end{align}
\end{lemma}
Notice that the equations above are holomorphic, as all functions appearing on the left hand sides contain no complex conjugate variables. In the smooth or formal case, these equations hold as equalities in the ring of formal power series \(\mathbb{C}[[z,w]]\).
\begin{proof} Put \(\Psi = g\phi+iff^t\). Setting \(\wba = 0\) in the mapping equation \eqref{e:me} yields
\begin{equation}\label{he1}
	i\, \overline{zA_{\alpha,\beta}z^t}\,\Psi(z,2iz\zba^t) - \overline{zz^t} \phi(z,2izz^t) - 2 \bar{z} f^t = 0.
\end{equation}
Substituting \(\zba_1 = -(iw + 2 z_2\zba_2)/(2z_1)\) into \eqref{he1} and clearing the denominator, we obtain an identity of the form
\begin{equation}\label{e:g}
	G(H(z,w),z,w,\zba_2) \equiv 0,
\end{equation}
which holds outside the varieties \(z_1 = 0\). Here, \(G(U,z,w,t)\) is polynomial in its arguments \(U,z,w,t\). The left-and side is viewed as a holomorphic function or formal power series of \(z_1,z_2, w, \zba_2\). Setting \(\zba_2 = 0\) in \eqref{e:g}, we obtain the first equality \eqref{h1}. Next, differentiating \eqref{e:g} with respect to \(\zba_2\) and setting \(\zba_2=0\), we obtain the third identity \eqref{h3}. Other identities can be obtain in the same manner and by exchanging the role of \(z_1\) and \(z_2\). We leave the details to the readers.
\end{proof}

Equations \eqref{h1}--\eqref{h5} are not independent; the last two equations can be obtained from the first three. Solving these equations simutanously, we obtain
\begin{lemma} Put \(\Psi = g\phi+iff^t\). Then
	\begin{equation}\label{e:s1}
		f = \frac{g}{w} z + \left(\frac{w\Psi}{2zz^t}\right) zA_{\alpha,\beta}, \quad \phi = -\frac{izA_{\alpha,\beta} z^t}{zz^t} \Psi .
	\end{equation}
\end{lemma}
\begin{proof}
We can rewrite the equations in Lemma~\ref{lem35} as a system of linear equations in 5 variables \(\Psi, f_1, f_2, \phi, g\). It turns out that the rank of the coefficient matrix over the quotient field of \(\mathbb{C}[[z,w]]\) is 3. Three equations \eqref{h1}, \eqref{h2}, and \eqref{h5} form the following system 
\begin{equation}\label{e:17}
	\left(
\begin{array}{ccccc}
	0 & 4 w z_2 & -i w^2 & -4 z_2^2 & \alpha  w^2 \\
	-4 w z_1 & 0 & i w^2 & 4 z_1^2 & \alpha  w^2 \\
	0 & 0 & -i zz^t & 0 & zA_{\alpha,\beta} z^t\\
\end{array}
\right)
\begin{pmatrix}
	f_1 \\ f_2 \\ \phi \\ g \\ \Psi 
\end{pmatrix}
= 
\begin{pmatrix}
	0 \\ 0\\ 0\\ 0 \\ 0
\end{pmatrix}.
\end{equation}
The reduced row-echelon form of the coefficient matrix can be computed easily, namely,
\begin{equation}\label{rref}
	\left(
	\begin{array}{ccccc}
		1 & 0 & 0 & -{z_1}/{w} & - w (\alpha  z_1+\beta  z_2)/(2zz^t) \\
		0 & 1 & 0 & -{z_2}/{w} & w (\alpha  z_2-\beta  z_1)/(2zz^t) \\
		0 & 0 & 1 & 0 & {i zA_{\alpha,\beta} z^t}/({zz^t}) \\
	\end{array}
	\right)
\end{equation}
From this, we can easily solve the system to obtain \eqref{e:s1}.
\end{proof}

To determine \(\phi\) and \(g\), we need another holomorphic equation. We prove the following lemma.
\begin{lemma} If \(\lambda = 0\), then
\begin{equation}\label{key}
	H_w(z,0) = \left(\frac{i}{2}zA_{\alpha,\beta} ,\ 0, \ 1\right).
\end{equation}
\end{lemma}
\begin{proof}
Consider the differential operator \(T = \partial/\partial \wba\). Applying \(L_1\) and \(T\) to the mapping equation consecutively and evaluating along \((\zba,\wba) = (0,0,0)\), we obtain
\[ 
f_{1w}(z,0) = \frac{i}{2} \left(\alpha z_1 + \beta z_2\right).
\]
Similarly, with \(L_1\) is replaced by \(L_2\), we obtain
\[ 
f_{2w}(z,0) = \frac{i}{2} \left(\beta z_1 - \alpha z_2\right).
\]
Differentiating \eqref{h5} with respect to \(w\), setting \(w=0\), and substituting, we obtain
\[ 
-iz z^t \phi_w(z,0) = 0.
\]
Finally, that \(g_w(z,0) = 1\) follows from Lemma~\ref{lem:32}. The proof is complete.
\end{proof}
\begin{lemma} If \(\lambda = 0\), then
\begin{equation}\label{e:s2}
		\alpha w(g\phi + i f f^t) - z_1(2+i\alpha w)f_1 - i\beta z_1 w f_2 + i w \phi + 2 z_1^2 = 0.
	\end{equation}
\end{lemma}
\begin{proof} Substituting \(\zba_1 = -iw/(2z_1)\) and \(\zba_2 = 0\) into the following equation, 
\[ 
L_1 (\rho(z,\bar{w}+2i z\zba^t), \overline{H(z,w)}) = 0,
\]
we obtain the desired formula. Details are left to the readers. 
\end{proof}

We continue the proof of the main result. Substituting \eqref{e:s1} into \eqref{e:s2}, clearing the denominator \(2w z z^t\), we obtain an equation of \(g\) and \(\Psi\) of the form
\begin{equation}\label{e:321}
	 2i z_1 (2i z_1 - w(\alpha z_1+\beta z_2)) g+z_1 w^2(2\alpha z_1 + 2 \beta z_2 - iwz_1 (\alpha^2+\beta^2))\Psi + 4wz_1^2 (z z^t) = 0.
\end{equation}
Substituting \eqref{e:s1} into the equation \(\Psi - (g\phi + i f f^t) = 0\) and clearing the common denominator \(-4i w^2 z z^t \), we obtain another equation of \(g\) and \(\phi\). Namely, we obtain
\begin{equation}\label{e:322}
	4(zz^t)^2g^2 + w^4(\alpha^2+\beta^2) \Psi^2 + 4i w^2 (z z^t) \Psi = 0.
\end{equation}
From \eqref{e:321} and \eqref{e:322}, we can uniquely solve for \((g,\Psi)\), which are holomorphic at the origin. The result is
\begin{equation}\label{e:g2}
	g(z,w) = \frac{4w}{4 - (\alpha^2+\beta^2)w^2},
	\quad
	\Psi(z,w) = \frac{4i z z^t}{4 - (\alpha^2+\beta^2)w^2}.
\end{equation}
Plugging these into \eqref{e:s1}, we obtain
\begin{equation}\label{e:phi}
\phi(z,w) = \frac{4 zA_{\alpha,\beta}z^t}{4 - (\alpha^2+\beta^2)w^2},\quad f(z,w) = \frac{2(2z + i wzA_{\alpha,\beta})}{4 - (\alpha^2+\beta^2)w^2},
\end{equation}
where
\begin{align*}
A_{\alpha,\beta} = \begin{pmatrix}
	\alpha & \beta \\ \beta & -\alpha
\end{pmatrix}.
\end{align*}
From this point on, we divide into two subcases:
\subsubsection*{Subcase 1: \((\alpha,\beta) = (0,0)\)}
It is easy to see that the map takes the form \((z,0,w)\), i.e., \(H\) is the linear map.

\subsubsection*{Subcase 2: \((\alpha,\beta)\ne (0,0)\)} In this case, \(r := \sqrt{\alpha^2 + \beta^2} > 0\). By composing with two suitable automorphisms of the source and target we rescale the matrix \(A\) to see that \(H\) is equivalent to a map of the form
\begin{equation}\label{e:rationalmap}
r(z,w) = \left(\frac{z(I+ i w A)}{1 - w^2},\ \frac{ 2zAz^t}{1 - w^2},\ \frac{w}{1 - w^2} \right),
\end{equation}
where \(I\) is the \(2\times 2\) identity matrix and 
\[A = A_{1,0} = 
\begin{pmatrix}
	1 & 0\\ 0 & - 1
\end{pmatrix}.
\]
Explicitly, let \(s\in [0,2\pi)\) be given by
\[ 
	\alpha = r \cos(s), \quad \beta = r\sin(s),
 \]
and consider the following automorphisms \(\gamma \in \Aut_0(\mathcal{X})\) and \(\psi \in \Aut_0(\mathbb{H}^5)\) given by

\begin{align*}
\gamma(z,\zeta,w) & = \left(\sqrt{r} z B,\, \zeta,\, rw\right), \\
\psi(z,w) & = \left(\sqrt{r}z B, r w\right),
\end{align*}
where \(B\) is the following \(2\times 2\) orthogonal matrix
\[ 
B =  \begin{pmatrix}
	\cos(s/2) & -\sin(s/2) \\ \sin(s/2) & \cos(s/2)
\end{pmatrix}.
\]
Then, by direct calculation, one can check that
\[
r = \gamma \circ H \circ \psi^{-1},
\]
which completes the proof of Case 1.
\subsection{Case 2:} \(\lambda \ne 0\). In this case, we will show that the map is irrational and an isometric embedding of the ``canonical'' K\"ahler--Einstein metrics from the Siegel domain into a one-sided neighborhood of \(\mathcal{X}\).
\begin{lemma}
If \(\lambda \ne 0\), then \(\alpha = \beta = 0\).
\end{lemma}
\begin{proof} Applying \(L_1^2\) the mapping equation \eqref{e:me} and evaluating along \(\zba_1 = \zba_2 = \wba = 0\), we obtain 
\begin{equation}\label{e:lem391}
	2z_1(\alpha + 4\bar{\nu}_1 z_1) f_1 + 2z_2(\beta + 4 \bar{\nu}_2 z_1) f_2 + (4 \bar{\sigma} z_1^2 + 2i \bar{\mu}_1 z_1 -\alpha) f f^t -(1-4i\bar{\lambda} z_1^2) \phi = 0,
\end{equation}
which holds along \(\Sigma\). Similarly, applying \(L_2\) and \(L_1\) consecutively to the mapping equation \eqref{e:me} and evaluating along \(\zba_1 = \zba_2 = \wba = 0\), we obtain 
\begin{multline}\label{e:lem392}
	(\beta z_1 + \alpha z_2 + 8 \bar{\nu}_1 z_1 z_2) f_1 + (-\alpha z_1 + \beta z_2 + 8 z_1 z_2 \bar{\nu}_2)f_2 \\ 
	+ (-\beta + i\bar{\mu}_1z_2 + i \bar{\mu}_2 z_1 + 4 \bar{\sigma} z_1 z_2) ff^t + 4i\bar{\lambda} z_1 z_2 \phi = 0
\end{multline}
along \(\Sigma\). Solving for \(\phi(z,0)\) from \eqref{e:lem391} and substituting the result into \eqref{e:lem392} and further substituting \eqref{e:lem31}, we obtain an equation of the form
\begin{equation}\label{e:324}
	M(z)\sqrt{1 - 4i \bar{\lambda} z z^t } + N(z) = 0,
\end{equation}
where \(M(z)\) and \(N(z)\) are explicit polynomials in \(z_1\) and \(z_2\). Since \(\lambda \ne 0\), equation \eqref{e:324} implies that \(M(z) = N(z) = 0\). The explicit formula of \(M(z)\), which is a quartic polynomial, is as follows:
\[ M(z) = 4 \bar{\lambda} \beta z_1^4 - 8 \bar{\lambda} \alpha z_1^3 z_2 + \cdots\]
where the dots represent monomials of different types, than the ones displayed. Equating the coefficients of \(z_1^4\) and \(z_1^3z_2\) to zero, we conclude that \(\alpha = \beta =0.\)
\end{proof}

We now complete the proof of Case 2. At an arbitrary point in \(p\in U\), a partial normal form of \(H\) at \(p\) must have coefficient \(\lambda\ne 0\), as otherwise, \(H\) were equivalent to a map appearing in Case 1. Therefore, \(H\) has vanishing geometric rank at an arbitrary point \(p\in U\) and hence on an open set. Thus $H$ extends to a local isometry by Theorem~\ref{thm:2.3}. By Xiao--Yuan \cite{xiao2020holomorphic}, it must be equivalent to \(\iota\). This completes Case 2. The case, when the map is nowhere transversal is treated in the next section, and it finishes the proof of Theorem~\ref{main}.

\subsection{Nowhere transversal maps} If the map $H$ is nowhere transversal in $\mathbb{H}^5$ then we have the following: 
\begin{lemma}
\label{lem:nowhere}
Any smooth CR map $H$, which sends $\mathbb{H}^5$ into $\mathcal{X}$ and is nowhere transversal in $\mathbb{H}^5$, is equivalent to a map of the form $(z,w) \mapsto (0, 0, \phi(z,w),0)$ for a smooth CR function $\phi$ fixing the origin.
\end{lemma}

\begin{proof}
The transitivity of the automorphisms of the source and target allows us to assume $H(0)=0$. Then the map has to satisfy the following equation:
\begin{align*}
(g(z,w) - \bar{g}(\bar{z}, \bar{w})) (1 - |\phi (z,w)|^2) - 2 i (f(z,w) \bar{f}^t (\bar{z}, \bar{w}) - \mathrm{Re}(f(z,w) f^t(z,w) \bar{\phi}(\bar{z}, \bar{w})) ) = 0,
\end{align*}
for all $z,w,\bar z, \bar w$. Setting $\bar z = \bar w = 0$ shows $g = 0$. We consider the weighted homogeneous expansion $f_j = \sum_{k\geq 0} f^j_k$, where $j=1,2$, and $\phi = \sum_{i\geq 0} \phi_i$, where $f_m^j$ and $\phi_m$  are weighted homogeneous polynomials of order $m$ with respect to the weight $1$ for $z_1, z_2$ and $2$ for $w$. Then we obtain when we collect terms of order $m$:
\begin{align*}
\sum_{k=0}^m f^1_k \bar{f}^1_{m-k} + \sum_{j=0}^m f^2_j \bar{f}^2_{m-j} + \mathrm{Re}\left( \sum_{i=0}^m \sum_{l=0}^i f^1_{m-i} f^1_{i-l} \bar{\phi}_l + \sum_{r=0}^m \sum_{s=0}^r f^2_{m-r} f^2_{r-s} \bar{\phi}_s  \right) =0.
\end{align*}
We show inductively that $f_n = 0$. $f_0=0$ follows from the fact that $H(0)=0$. Assuming $f_0 = \ldots = f_n = 0$ and consider $m = 2 (n+1)$ in the above equation. Then the equation becomes $f^1_{n+1} \bar f^1_{n+1 } + f^2_{n+1} \bar f^2_{n+1 } = 0$, which shows the claim.
\end{proof}

\section{Construction of maps, proof of 			Corollary~\ref{cor:balld4}, and an example}\label{sec:5}

From the well-known fact that the smooth part of the boundary of a type IV domain is locally CR equivalent to the tube over the future light cone and Theorem~\ref{main}, we can easily deduce Corollary~\ref{cor:balld4} about a characterization of proper holomorphic maps from the unit ball in $\CC^3$ into the type IV domain in \(\mathbb{C}^4\).

Before providing the proof of the corollary, we explain how to construct the maps stated in Corollary~\ref{cor:balld4} from the maps in Theorem~\ref{main}. Of course, the two isometric embeddings were known earlier from the work of Xiao--Yuan \cite{xiao2020holomorphic}. The third map was given in Reiter--Son \cite{rs22}. 

First, the rational map $\Phi$ given by
\begin{equation}\label{eq:phi}
\Phi(z,\zeta,w)
=
\left(\frac{2i z}{2i + w}, \ \frac{2i  - w - 2i \zeta -(w\zeta + izz^t)}{2(2i+w)},\ \frac{i\left(2i - w + 2i \zeta +(w\zeta + izz^t) \right)}{2(2i+w)}\right)
\end{equation}
sends a neighborhood $V$ of the origin biholomorphically onto some neighborhood $U$ of $p$ with $\Phi(0,0,0,0) = p$ and  $\Phi(\mathcal{X})  \subset \partial D^{\mathrm{IV}}_4$ (cf. \cite{rs22}). Composing it with the linear map, we obtain
\begin{equation}\label{eq:phiell}
\Phi \circ \ell (z,w)
=
\left(\frac{2i z}{2i + w}, \ \frac{2i  - w -izz^t}{2(2i+w)},\ \frac{i\left(2i - w + izz^t\right)}{2(2i+w)}\right).
\end{equation}
Next, consider the following modified Cayley transform:
\[ 
C_1(z,w) = \left(\frac{\sqrt{2}z}{1+w},\ \frac{2i(1-w)}{1+w}\right), 
\]
one can check that
\[ 
R_0 = \Phi \circ \ell \circ C_1
\]
is the well-known rational isometry. This map has a singular point at \((0,0,-1)\).

For the other rational map, we take
\[ 
A_{-1/2,0}= \begin{pmatrix}
-1/2 & 0 \\ 0 & 1/2
\end{pmatrix}
\]
and consider the map \(H_{A_{-1/2,0}}\). The (modified) Cayley transform 
\[ 
C_2(z_1,z_2,w) = \left(\frac{2z_1}{1-w},\ \frac{-2z_2}{1-w},\ \frac{4i(1+w)}{1-w}\right)
\]
sends the 3-ball into the Siegel domain and sends the sphere into the Heisenberg hypersurface. By direct calculations
\begin{equation}
\Phi \circ H_{A_{-1/2,0}} \circ C_2(z,w) = \left(z_1,z_2w, \dfrac{w^2-z_2^2}{2}, \dfrac{i(w^2+z_2^2)}{2}\right).
\end{equation}
The formulas of the automorphisms and maps to obtain the irrational map are quite complicated. Details are left to the readers.

Now we give the proof of the characterization of the proper holomorphic maps from \(\mathbb{B}^3\) into \(D^{\mathrm{IV}}_4\) that extend smoothly to a boundary point. The argument here is very similar to Reiter--Son \cite{rs22}.

\begin{proof}[Proof of Corollary~\ref{cor:balld4}]
Assume that $H\colon \mathbb{B}^3 \to \tfour$ is a proper holomorphic map that extends smoothly to a boundary point $p\in \mathbb{S}^5$. We can assume that $p=(0,0,1)$ and $H(p) = \left(0,0,\frac{1}{2},\frac{i}{2}\right)$. Then $\tilde{H}: = \Phi^{-1} \circ H \circ \mathcal{C}^{-1}$ defines a germ at the origin of CR maps sending $\mathbb{H}^5$ into $\lc$. Therefore, either it is equivalent to one of the transversal maps \(\ell, r,\) and \(\iota\) (by Theorem~\ref{main}), or it is nowhere transversal and of the form \((z,w) \mapsto (0,0,\phi(z,w),0)\) (by Corollary~\ref{lem:nowhere}). The last case cannot happen, as otherwise \(H\) is not proper. Hence, $\tilde{H}$ must belong to one of three equivalence classes of the germs of $\Phi^{-1} \circ F\circ \mathcal{C}^{-1}$ with $F\in \{R_0, P, I\}$. In fact, these three germs of CR maps $\Phi^{-1} \circ F\circ \mathcal{C}^{-1}$, $F\in \{R_0, P, I\}$ are pairwise inequivalent as can be easily checked using the Ahlfors invariant and the rationality. Thus, there are local CR automorphisms $\psi\in \Aut (\lc, 0)$ and $\gamma \in \Aut (\mathbb{H}^5, 0)$ such that $\psi \circ \tilde H \circ \gamma^{-1} = \Phi^{-1} \circ F\circ \mathcal{C}^{-1}$ near the origin for some $F\in \{R_0, P,  I\}$. Thus, if $\tilde \psi : = \Phi \circ \psi \circ \Phi^{-1}$ and $\tilde \gamma := \mathcal{C}^{-1} \circ \gamma \circ \mathcal{C}$, then, as germs at $p$, we have $\tilde \psi \circ H \circ \tilde \gamma^{-1} =F  \in \{R_0, P, I\}$. But $\tilde \psi$ is a global automorphism of $\tfour$ by a Alexander-type theorem of Mok--Ng \cite{mok2012germs} (see also Reiter--Son \cite[Theorem~2.3]{rs22} for a simpler proof in the special case of \(D^{\mathrm{IV}}_3\)) and $\tilde{\gamma}$ is a global automorphism of the unit ball. This completes the proof of Corollary~\ref{cor:balld4}.
\end{proof}

We end this paper by the following example of a family of maps.
\begin{example}
For \(s\in \mathbb{R}\), the family of proper holomorphic maps 
\begin{equation}
	P_B(z,w) = \left(z+(w-1)zB,\ \frac{1}{2}(w^2-zBz^t),\ \frac{1}{2i}(w^2+zBz^t)\right),
\end{equation}
where \(B = v^t v \in \mathrm{Mat}(2\times 2; \mathbb{R})\) and \(v = (\cos(s),\, \sin(s))\). For each $s$ this map sends \(\mathbb{B}^3\) into \(D^{\mathrm{IV}}_4\) properly. By Corollary~\ref{cor:balld4}, it is equivalent to \(P(z,w)\).
\end{example}

\end{document}